\documentclass[11pt]{article}

\usepackage[T1]{fontenc}
\usepackage[utf8]{inputenc}
\usepackage{lmodern}
\usepackage[english]{babel}
\usepackage{amsmath,amssymb,amsthm}
\usepackage{graphicx}
\usepackage{tikz}

\newtheorem{theorem}{Theorem}
\newtheorem{lemma}[theorem]{Lemma}
\newtheorem{proposition}[theorem]{Proposition}
\newtheorem{corollary}[theorem]{Corollary}
\theoremstyle{definition}
\newtheorem{fact}[theorem]{Fact}
\theoremstyle{remark}
\newtheorem{remark}[theorem]{Remark}

\usepackage{xcolor}
\definecolor{G}{rgb}{0.0, 0.5, 0.0}
\definecolor{B}{rgb}{0.4, 0.2, 0.1}
\definecolor{BL}{rgb}{0.0, 0.3, 1.0}
\definecolor{L}{rgb}{0.6, 0.0, 0.8}
\definecolor{O}{rgb}{0.8, 0.3, 0.0}
\definecolor{R}{rgb}{0.8, 0.0, 0.0}
\definecolor{LINK}{rgb}{0.35, 0.15, 0.1}
\definecolor{URL}{rgb}{0.6, 0.2, 0.7}

\usepackage{caption}
\usepackage[linkcolor=LINK, urlcolor=URL, citecolor=magenta]{hyperref}
\usepackage[all]{hypcap}
\hypersetup{linktocpage,colorlinks=true,pdfborder={0 0 0}}

\usepackage{geometry}
\usepackage{booktabs}
\newcommand{\C}{\mathbb{C}}
\newcommand{\adj}{\sim}

\newcommand{\w}{\omega}
\newcommand{\wb}{\bar{\omega}}

\title{Harborth's theorem $p(5)=18$ revisited\\[5mm]}

\author{Mike Winkler\\[5mm]
	Faculty of Mathematics\\Ruhr University Bochum, Germany\\
	mike.winkler@ruhr-uni-bochum.de}

\date{October 28, 2026}

\begin{document}
	\maketitle
	
	\begin{abstract}
		Let $p(n)$ be the minimum cardinality of a finite planar point set in which each point has exactly $n$ unit-distance neighbours. Harborth proved $p(5)=18$ in work presented in 1985 and published in 1987, using a detailed case analysis. We revisit the proof using convex-hull arguments, incidence counts, and elementary properties of intersections of unit circles. These tools give an alternative treatment of configurations with $14$ and $16$ points and organize the final case distinction by a common vector-sum structure. We also treat a configuration not explicitly covered by the case analysis in the online PDF copy of Harborth's paper available to us, and give an explicit justification of the six-point classification used in the argument. No result subsequent to Harborth's proof is required.
	\end{abstract}
	

	\section{Introduction}
	
	A finite set $S\subset\C$ is called \emph{$n$-regular} if every point of $S$ has distance $1$ to exactly $n$ other points of $S$. Let $p(n)$ denote the minimum cardinality of an $n$-regular set.
	
	\begin{theorem}[Harborth \cite{H}]\label{thm:main}
		$p(5)=18$.
	\end{theorem}
	
	The first values are $p(1)=2$, $p(2)=3$, $p(3)=6$, $p(4)=9$, and $p(5)=18$. Figure~\ref{fig:small-regular} shows realizations for $n=3,4,5$. These examples illustrate the Minkowski-sum constructions used for the upper bounds.
	
	\begin{figure}[htbp]
		\centering
		\begin{tikzpicture}[x=1cm,y=1cm]
			\coordinate (p3v0) at (0.891103,1.471103);
			\coordinate (p3v1) at (2.068897,0.791103);
			\coordinate (p3v2) at (2.251103,1.471103);
			\coordinate (p3v3) at (3.428897,0.791103);
			\coordinate (p3v4) at (1.571103,2.648897);
			\coordinate (p3v5) at (2.748897,1.968897);
			\draw[line width=0.42pt] (p3v0) -- (p3v1);
			\draw[line width=0.42pt] (p3v0) -- (p3v2);
			\draw[line width=0.42pt] (p3v0) -- (p3v4);
			\draw[line width=0.42pt] (p3v1) -- (p3v3);
			\draw[line width=0.42pt] (p3v1) -- (p3v5);
			\draw[line width=0.42pt] (p3v2) -- (p3v3);
			\draw[line width=0.42pt] (p3v2) -- (p3v4);
			\draw[line width=0.42pt] (p3v3) -- (p3v5);
			\draw[line width=0.42pt] (p3v4) -- (p3v5);
			\fill (p3v0) circle (1.10pt);
			\fill (p3v1) circle (1.10pt);
			\fill (p3v2) circle (1.10pt);
			\fill (p3v3) circle (1.10pt);
			\fill (p3v4) circle (1.10pt);
			\fill (p3v5) circle (1.10pt);
			\node[font=\small] at (2.160,-0.35) {(a) $p(3)=6$};
			\coordinate (p4v0) at (5.485840,0.245840);
			\coordinate (p4v1) at (6.854160,1.035840);
			\coordinate (p4v2) at (5.485840,1.825840);
			\coordinate (p4v3) at (7.065840,0.245840);
			\coordinate (p4v4) at (8.434160,1.035840);
			\coordinate (p4v5) at (7.065840,1.825840);
			\coordinate (p4v6) at (6.275840,1.614160);
			\coordinate (p4v7) at (7.644160,2.404160);
			\coordinate (p4v8) at (6.275840,3.194160);
			\draw[line width=0.42pt] (p4v0) -- (p4v1);
			\draw[line width=0.42pt] (p4v0) -- (p4v2);
			\draw[line width=0.42pt] (p4v0) -- (p4v3);
			\draw[line width=0.42pt] (p4v0) -- (p4v6);
			\draw[line width=0.42pt] (p4v1) -- (p4v2);
			\draw[line width=0.42pt] (p4v1) -- (p4v4);
			\draw[line width=0.42pt] (p4v1) -- (p4v7);
			\draw[line width=0.42pt] (p4v2) -- (p4v5);
			\draw[line width=0.42pt] (p4v2) -- (p4v8);
			\draw[line width=0.42pt] (p4v3) -- (p4v4);
			\draw[line width=0.42pt] (p4v3) -- (p4v5);
			\draw[line width=0.42pt] (p4v3) -- (p4v6);
			\draw[line width=0.42pt] (p4v4) -- (p4v5);
			\draw[line width=0.42pt] (p4v4) -- (p4v7);
			\draw[line width=0.42pt] (p4v5) -- (p4v8);
			\draw[line width=0.42pt] (p4v6) -- (p4v7);
			\draw[line width=0.42pt] (p4v6) -- (p4v8);
			\draw[line width=0.42pt] (p4v7) -- (p4v8);
			\fill (p4v0) circle (1.10pt);
			\fill (p4v1) circle (1.10pt);
			\fill (p4v2) circle (1.10pt);
			\fill (p4v3) circle (1.10pt);
			\fill (p4v4) circle (1.10pt);
			\fill (p4v5) circle (1.10pt);
			\fill (p4v6) circle (1.10pt);
			\fill (p4v7) circle (1.10pt);
			\fill (p4v8) circle (1.10pt);
			\node[font=\small] at (6.960,-0.35) {(b) $p(4)=9$};
			\coordinate (p5v0) at (11.507338,0.737399);
			\coordinate (p5v1) at (10.285737,0.292773);
			\coordinate (p5v2) at (12.517627,1.555516);
			\coordinate (p5v3) at (11.296027,1.110889);
			\coordinate (p5v4) at (11.303973,2.021394);
			\coordinate (p5v5) at (10.082373,1.576768);
			\coordinate (p5v6) at (12.807338,0.737399);
			\coordinate (p5v7) at (11.585737,0.292773);
			\coordinate (p5v8) at (13.817627,1.555516);
			\coordinate (p5v9) at (12.596027,1.110889);
			\coordinate (p5v10) at (12.603973,2.021394);
			\coordinate (p5v11) at (11.382373,1.576768);
			\coordinate (p5v12) at (12.157338,1.863232);
			\coordinate (p5v13) at (10.935737,1.418606);
			\coordinate (p5v14) at (13.167627,2.681349);
			\coordinate (p5v15) at (11.946027,2.236723);
			\coordinate (p5v16) at (11.953973,3.147227);
			\coordinate (p5v17) at (10.732373,2.702601);
			\draw[line width=0.42pt] (p5v0) -- (p5v1);
			\draw[line width=0.42pt] (p5v0) -- (p5v2);
			\draw[line width=0.42pt] (p5v0) -- (p5v4);
			\draw[line width=0.42pt] (p5v0) -- (p5v6);
			\draw[line width=0.42pt] (p5v0) -- (p5v12);
			\draw[line width=0.42pt] (p5v1) -- (p5v3);
			\draw[line width=0.42pt] (p5v1) -- (p5v5);
			\draw[line width=0.42pt] (p5v1) -- (p5v7);
			\draw[line width=0.42pt] (p5v1) -- (p5v13);
			\draw[line width=0.42pt] (p5v2) -- (p5v3);
			\draw[line width=0.42pt] (p5v2) -- (p5v4);
			\draw[line width=0.42pt] (p5v2) -- (p5v8);
			\draw[line width=0.42pt] (p5v2) -- (p5v14);
			\draw[line width=0.42pt] (p5v3) -- (p5v5);
			\draw[line width=0.42pt] (p5v3) -- (p5v9);
			\draw[line width=0.42pt] (p5v3) -- (p5v15);
			\draw[line width=0.42pt] (p5v4) -- (p5v5);
			\draw[line width=0.42pt] (p5v4) -- (p5v10);
			\draw[line width=0.42pt] (p5v4) -- (p5v16);
			\draw[line width=0.42pt] (p5v5) -- (p5v11);
			\draw[line width=0.42pt] (p5v5) -- (p5v17);
			\draw[line width=0.42pt] (p5v6) -- (p5v7);
			\draw[line width=0.42pt] (p5v6) -- (p5v8);
			\draw[line width=0.42pt] (p5v6) -- (p5v10);
			\draw[line width=0.42pt] (p5v6) -- (p5v12);
			\draw[line width=0.42pt] (p5v7) -- (p5v9);
			\draw[line width=0.42pt] (p5v7) -- (p5v11);
			\draw[line width=0.42pt] (p5v7) -- (p5v13);
			\draw[line width=0.42pt] (p5v8) -- (p5v9);
			\draw[line width=0.42pt] (p5v8) -- (p5v10);
			\draw[line width=0.42pt] (p5v8) -- (p5v14);
			\draw[line width=0.42pt] (p5v9) -- (p5v11);
			\draw[line width=0.42pt] (p5v9) -- (p5v15);
			\draw[line width=0.42pt] (p5v10) -- (p5v11);
			\draw[line width=0.42pt] (p5v10) -- (p5v16);
			\draw[line width=0.42pt] (p5v11) -- (p5v17);
			\draw[line width=0.42pt] (p5v12) -- (p5v13);
			\draw[line width=0.42pt] (p5v12) -- (p5v14);
			\draw[line width=0.42pt] (p5v12) -- (p5v16);
			\draw[line width=0.42pt] (p5v13) -- (p5v15);
			\draw[line width=0.42pt] (p5v13) -- (p5v17);
			\draw[line width=0.42pt] (p5v14) -- (p5v15);
			\draw[line width=0.42pt] (p5v14) -- (p5v16);
			\draw[line width=0.42pt] (p5v15) -- (p5v17);
			\draw[line width=0.42pt] (p5v16) -- (p5v17);
			\fill (p5v0) circle (1.10pt);
			\fill (p5v1) circle (1.10pt);
			\fill (p5v2) circle (1.10pt);
			\fill (p5v3) circle (1.10pt);
			\fill (p5v4) circle (1.10pt);
			\fill (p5v5) circle (1.10pt);
			\fill (p5v6) circle (1.10pt);
			\fill (p5v7) circle (1.10pt);
			\fill (p5v8) circle (1.10pt);
			\fill (p5v9) circle (1.10pt);
			\fill (p5v10) circle (1.10pt);
			\fill (p5v11) circle (1.10pt);
			\fill (p5v12) circle (1.10pt);
			\fill (p5v13) circle (1.10pt);
			\fill (p5v14) circle (1.10pt);
			\fill (p5v15) circle (1.10pt);
			\fill (p5v16) circle (1.10pt);
			\fill (p5v17) circle (1.10pt);
			\node[font=\small] at (11.950,-0.35) {(c) $p(5)=18$};
		\end{tikzpicture}
		\caption{Regular unit-distance sets with $6$, $9$, and $18$ points, realizing $p(3)=6$, $p(4)=9$, and $p(5)=18$, respectively. The examples are obtained as generic Minkowski sums of unit triangles and, for odd degree, a unit segment. Only unit-distance pairs are drawn.}
		\label{fig:small-regular}
	\end{figure}
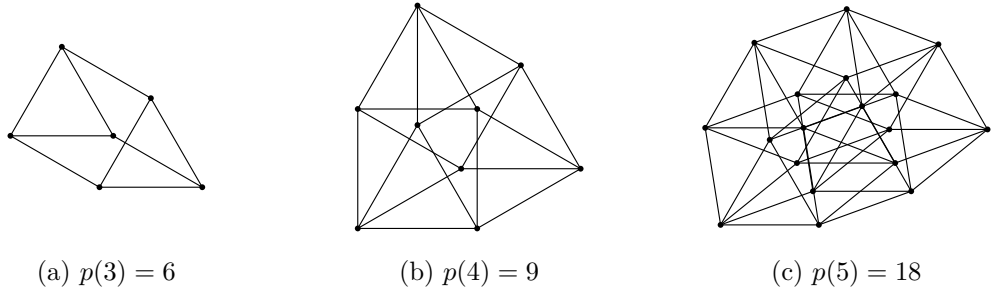
	
	The upper bound follows from $p(n)\le p(r)p(n-r)$, obtained by taking a Minkowski sum after a generic relative rotation, together with $p(1)=2$ and $p(4)\le p(2)^2=9$. For finite point sets, all but finitely many relative rotations make the sum map injective and create no new unit distances: for fixed nonzero vectors $u,v$, the equation $|u+\rho v|=1$ has only finitely many solutions in the rotation $\rho$. We prove that no $5$-regular set has at most $16$ points. Since a $5$-regular graph has an even number of vertices, this gives the lower bound $p(5)\ge18$.
	
	\pagebreak[3]
	
	While examining an online PDF copy of Harborth's paper~\cite{H}, we found that the numbering of the subcases in Case 6 passes directly from 6.6 to 6.8, and we could not locate an explicit treatment of one configuration (Remark~\ref{rem:missing-case}). Attempting to exclude this configuration led to a reorganization of the argument around missing vector-sum positions. The alternative proof emphasizes explicit arguments and common local counting rules rather than brevity.
	
	A convex-hull argument and incidence counting reduce the lower-bound problem to $14$ or $16$ points; the former is excluded directly (Section~\ref{sec:reduction}). For $16$ points, triangle counting forces a point in at least three unit triangles. Its neighbours contain either four consecutive hexagon directions or a three-vertex path and a disjoint edge. The first configuration is excluded by circle intersections and a six-point classification (Section~\ref{sec:hexagon}); the second is reduced to local counts on seven candidate sum positions, six of which form a $2\times3$ ladder (Section~\ref{sec:path}).

	\section{Preliminaries}
	
	We identify the plane with $\C$ and put $\w=e^{i\pi/3}$. For $x,y\in\C$ we write $x\adj y$ if $|x-y|=1$. For $v\in S$, let $t(v)$ be the number of unit triangles with vertex $v$, equivalently the number of unit distances among the five neighbours of $v$.
	
	We repeatedly use the following elementary facts.
	\begin{itemize}
		\item[(P1)] Two distinct points $x,y$ have at most two common points at unit distance. If $p$ is one of them, then $x+y-p$ is the other; the two coincide if and only if $|x-y|=2$.
		\item[(P2)] If at most five points lie on a unit circle with centre $p$, the graph of unit distances among them is a disjoint union of paths; each path consists of consecutive vertices of a regular hexagon centred at $p$.
		\item[(P3)] Four points span at most five unit distances. Five occur only for two unit triangles with a common edge; four occur only for a rhombus or for a unit triangle with a pendant point.
	\end{itemize}
	
	For later use write a lattice point $p+q\w$ as $(p,q)$ and define
	\begin{align*}
		T_1 & =\{(0,0),(1,0),(0,1),(-1,1),(-1,0),(0,-1)\}, \\
		T_2 & =\{(-1,0),(0,0),(1,0),(-1,1),(0,1),(-1,2)\}, \\
		T_3 & =\{(-1,0),(0,0),(1,0),(-1,1),(0,1),(1,1)\}.
	\end{align*}
	A \emph{prism} is a set $\Delta\cup(\Delta+v)$, where $\Delta$ is a unit triangle and $|v|=1$.
	
	\begin{fact}[Six-point classification]\label{fact:six}
		A set of six distinct points with exactly nine unit distances is congruent to $T_1$, $T_2$ or $T_3$, or it is a prism.
	\end{fact}
	
	\begin{proof}
		A unit-distance graph contains no $K_{2,3}$, by (P1). It contains no $K_4$: for two adjacent vertices the two common unit-distance points are the two vertices of the equilateral triangles on their edge, and these two points are at distance $\sqrt3$, not $1$. Inspection of the $21$ graphs with six vertices and nine edges in Harary's atlas \cite{Ha} leaves exactly five which contain neither $K_4$ nor $K_{2,3}$. Their degree sequences are $(5,3,3,3,2,2)$, $(4,4,4,2,2,2)$, $(4,4,3,3,2,2)$, $(3,3,3,3,3,3)$, and $(4,3,3,3,3,2)$. The first four are the graphs of $T_1,T_2,T_3$ and the prism, respectively (Figure~\ref{fig:six-point}). The fifth consists of three successive unit triangles $vab,vbc,vcd$, together with a sixth vertex adjacent to $a$ and $d$. Since the points are distinct, the four points $a,b,c,d$ are consecutive vertices of the unit hexagon about $v$, so $|a-d|=2$. Their two unit circles are tangent at $v$, forcing the sixth vertex to coincide with $v$, a contradiction.
		
		For $T_1,T_2,T_3$, the graph is obtained by successively attaching unit triangles along already fixed edges; with six distinct points the incidence pattern forces each attachment, so the realization is unique up to congruence. In the prism graph, (P1) makes the three quadrilaterals rhombi, so every realization has the form $\Delta\cup(\Delta+v)$ with $\Delta$ a unit triangle and $|v|=1$. Every prism vertex has degree $3$ inside the prism. Moreover any two distinct prism points have distance less than $2$: if $x,y$ are vertices of the same unit triangle and $|x-(y+v)|=2$, equality in the triangle inequality forces $x-y=-v$, hence $x+v=y$, so two of the six prism points coincide.
	\end{proof}
	
	\begin{figure}[htbp]
		\centering
		\begin{tikzpicture}[x=1cm,y=1cm]
			\begin{scope}[shift={(0.00,0)}]
				\draw[line width=0.55pt] (1.2300,1.2000) -- (2.2100,1.2000);
				\draw[line width=0.55pt] (1.2300,1.2000) -- (1.7200,2.0487);
				\draw[line width=0.55pt] (1.2300,1.2000) -- (0.7400,2.0487);
				\draw[line width=0.55pt] (1.2300,1.2000) -- (0.2500,1.2000);
				\draw[line width=0.55pt] (1.2300,1.2000) -- (0.7400,0.3513);
				\draw[line width=0.55pt] (2.2100,1.2000) -- (1.7200,2.0487);
				\draw[line width=0.55pt] (1.7200,2.0487) -- (0.7400,2.0487);
				\draw[line width=0.55pt] (0.7400,2.0487) -- (0.2500,1.2000);
				\draw[line width=0.55pt] (0.2500,1.2000) -- (0.7400,0.3513);
				\fill (1.2300,1.2000) circle (1.6pt);
				\fill (2.2100,1.2000) circle (1.6pt);
				\fill (1.7200,2.0487) circle (1.6pt);
				\fill (0.7400,2.0487) circle (1.6pt);
				\fill (0.2500,1.2000) circle (1.6pt);
				\fill (0.7400,0.3513) circle (1.6pt);
				\node[font=\small] at (1.23,-0.08) {$T_1$};
			\end{scope}
			\begin{scope}[shift={(3.13,0)}]
				\draw[line width=0.55pt] (0.2500,0.3513) -- (1.2300,0.3513);
				\draw[line width=0.55pt] (0.2500,0.3513) -- (0.7400,1.2000);
				\draw[line width=0.55pt] (1.2300,0.3513) -- (2.2100,0.3513);
				\draw[line width=0.55pt] (1.2300,0.3513) -- (0.7400,1.2000);
				\draw[line width=0.55pt] (1.2300,0.3513) -- (1.7200,1.2000);
				\draw[line width=0.55pt] (2.2100,0.3513) -- (1.7200,1.2000);
				\draw[line width=0.55pt] (0.7400,1.2000) -- (1.7200,1.2000);
				\draw[line width=0.55pt] (0.7400,1.2000) -- (1.2300,2.0487);
				\draw[line width=0.55pt] (1.7200,1.2000) -- (1.2300,2.0487);
				\fill (0.2500,0.3513) circle (1.6pt);
				\fill (1.2300,0.3513) circle (1.6pt);
				\fill (2.2100,0.3513) circle (1.6pt);
				\fill (0.7400,1.2000) circle (1.6pt);
				\fill (1.7200,1.2000) circle (1.6pt);
				\fill (1.2300,2.0487) circle (1.6pt);
				\node[font=\small] at (1.23,-0.08) {$T_2$};
			\end{scope}
			\begin{scope}[shift={(6.26,0)}]
				\draw[line width=0.55pt] (0.0050,0.7756) -- (0.9850,0.7756);
				\draw[line width=0.55pt] (0.0050,0.7756) -- (0.4950,1.6244);
				\draw[line width=0.55pt] (0.9850,0.7756) -- (1.9650,0.7756);
				\draw[line width=0.55pt] (0.9850,0.7756) -- (0.4950,1.6244);
				\draw[line width=0.55pt] (0.9850,0.7756) -- (1.4750,1.6244);
				\draw[line width=0.55pt] (1.9650,0.7756) -- (1.4750,1.6244);
				\draw[line width=0.55pt] (1.9650,0.7756) -- (2.4550,1.6244);
				\draw[line width=0.55pt] (0.4950,1.6244) -- (1.4750,1.6244);
				\draw[line width=0.55pt] (1.4750,1.6244) -- (2.4550,1.6244);
				\fill (0.0050,0.7756) circle (1.6pt);
				\fill (0.9850,0.7756) circle (1.6pt);
				\fill (1.9650,0.7756) circle (1.6pt);
				\fill (0.4950,1.6244) circle (1.6pt);
				\fill (1.4750,1.6244) circle (1.6pt);
				\fill (2.4550,1.6244) circle (1.6pt);
				\node[font=\small] at (1.23,-0.08) {$T_3$};
			\end{scope}
			\begin{scope}[shift={(9.39,0)},scale=1.12]
				\draw[line width=0.55pt] (0.2797,0.9437) -- (1.2597,0.9437);
				\draw[line width=0.55pt] (0.2797,0.9437) -- (0.7697,1.7924);
				\draw[line width=0.55pt] (0.2797,0.9437) -- (1.2003,0.6076);
				\draw[line width=0.55pt] (1.2597,0.9437) -- (0.7697,1.7924);
				\draw[line width=0.55pt] (1.2597,0.9437) -- (2.1803,0.6076);
				\draw[line width=0.55pt] (0.7697,1.7924) -- (1.6903,1.4563);
				\draw[line width=0.55pt] (1.2003,0.6076) -- (2.1803,0.6076);
				\draw[line width=0.55pt] (1.2003,0.6076) -- (1.6903,1.4563);
				\draw[line width=0.55pt] (2.1803,0.6076) -- (1.6903,1.4563);
				\fill (0.2797,0.9437) circle (1.6pt);
				\fill (1.2597,0.9437) circle (1.6pt);
				\fill (0.7697,1.7924) circle (1.6pt);
				\fill (1.2003,0.6076) circle (1.6pt);
				\fill (2.1803,0.6076) circle (1.6pt);
				\fill (1.6903,1.4563) circle (1.6pt);
				\node[font=\small] at (1.23,-0.08) {prism};
			\end{scope}
		\end{tikzpicture}
		\caption{The four types of six-point sets with exactly nine unit distances. Up to congruence, the prism $\Delta\cup(\Delta+v)$ depends on the direction of $v$; one member of this family is shown. Only unit-distance pairs are drawn.}
		\label{fig:six-point}
	\end{figure}
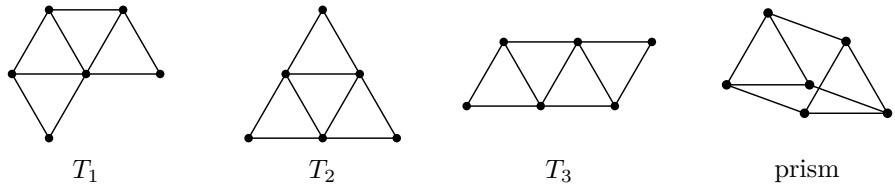
	
	\subsection{Local notation and rhombus rules}
	
	Let $N=|S|$. Fix $P\in S$, translate so that $P=0$, and let $Z=\{z_1,\ldots,z_5\}$ be its five neighbours. Put $e=t(0)$ and $O=S\setminus(\{0\}\cup Z)$. For $i\ne j$ write $s_{ij}=z_i+z_j$. If $s_{ij}=s_{kl}\ne0$, then $\{i,j\}=\{k,l\}$: the two unit vectors in either pair are precisely the intersection points of the unit circles centred at $0$ and at the common sum. Thus distinct index pairs give distinct sum points in $O$.
	
	A point of $O$ has at most two neighbours in $Z$. If it has two, say $z_i,z_j$, then it equals $s_{ij}$ by (P1). We call such a point a \emph{sum point} or a point of type A. Points with exactly one or no neighbour in $Z$ are of type B or C. Let their numbers be $a,b,c$. In the counting arguments below, a point outside a specified family of sum points is said to \emph{supply} a neighbour of a sum point when the two are at unit distance; this is only shorthand for a unit-distance incidence. Counting the unit distances between $Z$ and $O$ gives
	\begin{equation}\label{eq:abc-general}
		2a+b=20-2e,\qquad a+b+c=N-6,
	\end{equation}
	and therefore
	\begin{equation}\label{eq:ac-general}
		a-c=26-N-2e.
	\end{equation}
	
	\begin{lemma}[Rhombus rules]\label{lem:rules}
		Let $i,j,k,r$ denote distinct indices.
		\begin{itemize}
			\item[(i)] $|s_{ki}-s_{kj}|=|z_i-z_j|$.
			\item[(ii)] If $y\ne z_k$, $y\adj s_{ki}$ and $y\adj s_{kj}$, then $y=z_k+z_i+z_j$.
			\item[(iii)] Let $s_{ij}\in O$ and $z_r\adj z_i$. If $y\adj z_r$ and $y\adj s_{ij}$, then $y\in\{z_i,s_{rj}\}$.
			\item[(iv)] Let $s_{ij}\in O$. The common unit-distance points of $z_i$ and $s_{ij}$ are $z_i+z_j\w$ and $z_i+z_j\wb$.
			\item[(v)] Suppose $z_i+z_j=z_m$ with $m\ne k$. If $y\adj s_{ik}$ and $y\adj s_{jk}$, then $y\in\{z_k,s_{mk}\}$.
			\item[(vi)] Let $y$ be of type B, with unique neighbour $z_r$ in $Z$, and suppose $y\adj s_{ij}$. If $r\notin\{i,j\}$, then $z_r$ is adjacent to neither $z_i$ nor $z_j$.
		\end{itemize}
	\end{lemma}
	
	\begin{proof}
		Part (i) is immediate. For (ii), $z_k$ is one common unit-distance point of $s_{ki}$ and $s_{kj}$, so (P1) gives the other as $s_{ki}+s_{kj}-z_k=z_k+z_i+z_j$. Part (iii) follows in the same way from the two common points $z_i$ and $s_{rj}$. Part (iv) is the elementary intersection of the two unit circles. Part (v) follows from (ii), since $z_k+z_i+z_j=s_{mk}$. For (vi), if for instance $z_r\adj z_i$, then (iii) gives $y\in\{z_i,s_{rj}\}$. The first point belongs to $Z$ and the second, if it belongs to $O$, is a sum point; neither can be of type B. The same argument applies to $z_j$.
	\end{proof}
	
	\begin{lemma}\label{lem:B}
		There is at most one pair $\{i,j\}$ such that
		\[
		z_i+z_j=z_r+z_s+z_t,
		\]
		where $\{r,s,t\}$ is the complementary three-element index set.
	\end{lemma}
	
	\begin{proof}
		Put $\Sigma=z_1+\cdots+z_5$. The identity is equivalent to $2(z_i+z_j)=\Sigma$. If it holds for two different pairs, then these pairs have the same sum $\Sigma/2$. If $\Sigma\ne0$, the unordered pair of unit vectors with a prescribed nonzero sum is unique by (P1). If $\Sigma=0$, each such pair is antipodal; two different pairs either share an index, forcing two $z_i$ to coincide, or are disjoint, forcing the fifth vector to be zero. Both are impossible.
	\end{proof}
	
	\begin{lemma}[Five-sum lemma]\label{lem:five-sums}
		A point $y\notin Z$ can be adjacent to at most four distinct sum points $s_{ij}$.
	\end{lemma}
	
	\begin{proof}
		Make a graph on the indices $1,\ldots,5$ and join $i$ to $j$ when $s_{ij}$ is a sum point adjacent to $y$. For fixed $i$, the corresponding points $s_{ij}$ are common unit-distance points of $z_i$ and $y$, so the index graph has maximum degree at most $2$ by (P1). If it had at least five edges, it would therefore be a $5$-cycle. Relabel so that $12,23,34,45,15$ are its edges. By Lemma~\ref{lem:rules}(ii), applied at indices $2$ and $1$,
		\[
		y=z_1+z_2+z_3=z_1+z_2+z_5,
		\]
		so $z_3=z_5$, a contradiction.
	\end{proof}
	
	\begin{lemma}[Eisenstein circle observation]\label{lem:eisenstein-circle}
		Let $L=\mathbb Z[\w]$, let $|\eta|=1$, and suppose that $\eta$ is not a sixth root of unity. If $u\in L$ and
		\[
		|u|\in\{1,\sqrt3,2,\sqrt7,3\},
		\]
		then $|\eta+u|\ne1$.
	\end{lemma}
	
	\begin{proof}
		If $|u|>2$, the assertion follows from the triangle inequality. Otherwise $|\eta+u|=1$ is equivalent to
		\[
		\operatorname{Re}(\eta\overline u)=-\frac{|u|^2}{2}.
		\]
		For $|u|=1$, the vector $u$ is a unit of $L$ and the angle from $u$ to $\eta$ is $\pm2\pi/3$. For $|u|=\sqrt3$, the direction of $u$ is $\pi/6$ modulo $\pi/3$, while the angle from $u$ to $\eta$ is $\pm5\pi/6$. For $|u|=2$, equality in the real-part estimate forces $\eta=-u/2$. In every case $\eta$ is a sixth root of unity, a contradiction.
	\end{proof}

	\section{Elementary reduction to sixteen points}\label{sec:reduction}
	
	\begin{proposition}\label{prop:reduction}
		If a $5$-regular set $S$ has at most $16$ points, then $|S|=16$.
	\end{proposition}
	
	\begin{proof}
		Choose a vertex $0$ of the convex hull of $S$. Its five neighbours lie in a cone of angle strictly less than $\pi$. Hence, by (P2), they determine at most three unit distances, so $e\le3$. The five neighbours require $20-2e$ incidences with the remaining $N-6$ points, each of which supplies at most two. Thus
		\[
		20-2e\le2(N-6),\qquad N\ge16-e\ge13.
		\]
		Since $5N$ is even, $N$ is even. Hence $N\in\{14,16\}$.
		
		Assume $N=14$. Then $|O|=8$ and
		\[
		a-c=12-2e.
		\]
		The preceding lower bound gives $e\ge2$.
		
		\emph{Case $e=2$.} If the two unit distances in $Z$ form a path $z_1\adj z_2\adj z_3$, then $s_{13}=z_2$. From \eqref{eq:abc-general} we get $(a,b,c)=(8,0,0)$. There are nine possible sum points outside $\{0\}\cup Z$, so exactly one would be absent. But $z_2$ already has the neighbours $0,z_1,z_3$ and can have only two neighbours in $O$; among $s_{12},s_{23},s_{24},s_{25}$ at least two must therefore be absent. Contradiction.
		
		It remains that the two unit distances are disjoint, say $z_1\adj z_2$ and $z_3\adj z_4$. Again $(a,b,c)=(8,0,0)$. The eight points of $O$ are sum points and each has degree $3$ inside $O$, so they span $12$ unit distances. Sum points with a common index can be adjacent only through one of the two edges of $Z$, by Lemma~\ref{lem:rules}(i); hence there are at most six such unit distances. There must therefore be at least six unit distances between sum points with disjoint index pairs.
		
		For a sum point $s_{ij}$ let $q(s_{ij})$ be the number of its neighbours whose index pair is disjoint from $\{i,j\}$. If $q(s_{ij})\ge2$, two of these neighbours share an index, and Lemma~\ref{lem:rules}(ii) gives
		\[
		z_i+z_j=\sum_{r\notin\{i,j\}}z_r.
		\]
		By Lemma~\ref{lem:B}, this can happen for at most one sum point. Hence the sum of all $q$-values is at most $3+7=10$, so there are at most five unit distances between disjoint index pairs. Contradiction.
		
		\emph{Case $e=3$.} Since all five directions lie in an angle strictly less than $\pi$, the three edges in $Z$ form a path of length two and a disjoint edge. Relabel so that
		\[
		z_1\adj z_2\adj z_3,\qquad z_2=z_1+z_3,\qquad z_4\adj z_5.
		\]
		A relation $s_{ij}=z_k$ is equivalent to $z_i\adj z_k\adj z_j$, and no two $z_i$ are antipodal because they lie in a cone of angle less than $\pi$. Hence the only sum $s_{ij}$ that coincides with a point of $Z$ is $s_{13}=z_2$, and no sum is $0$. Thus there are nine possible sum points in $O$. The count gives
		\[
		(a,b,c)=(7,0,1)\quad\hbox{or}\quad(6,2,0).
		\]
		In the first alternative, the unique type C point has all five neighbours among the seven sum points, contrary to Lemma~\ref{lem:five-sums}.
		
		Consider $(a,b,c)=(6,2,0)$. Exactly three of the nine possible sum points are absent. Let $M$ be their index pairs, and let $m_i$ count the members of $M$ containing $i$. Counting the missing sum neighbours at the five points $z_i$ gives
		\[
		m_2\ge2,\qquad m_4\ge1,\qquad m_5\ge1.
		\]
		Let $E$ be the number of unit distances among the six sum points and let $h$ be $1$ or $0$ according as the two type B points are adjacent or not. Degree counting in $O$ gives
		\[
		18=2E+F,\qquad 8=F+2h,
		\]
		and therefore
		\begin{equation}\label{eq:E14}
			E=5+h\in\{5,6\}.
		\end{equation}
		
		Among sum points with disjoint index pairs, all adjacencies are impossible by Lemma~\ref{lem:rules}(iii), except possibly
		\[
		s_{12}\adj s_{45},\qquad s_{23}\adj s_{45}.
		\]
		These two cannot both occur: a common unit-distance point of $s_{12}$ and $s_{23}$ is either $z_2$ or $2z_2$, whereas $s_{45}\in O$ and $|s_{45}|=\sqrt3$.
		
		Of the $\binom{9}{3}$ three-element subsets of eligible sum positions, exactly $20$ satisfy $m_2\ge2$ and $m_4,m_5\ge1$. Identifying subsets under the relabellings $1\leftrightarrow3$ and $4\leftrightarrow5$ gives the following seven possibilities. Their orbit sizes, in table order, are $4,1,2,4,4,4,1$, which sum to $20$. The middle column counts the forced unit distances between present sum points with a common index; the last column includes the at most one exceptional adjacency just described.
		\begin{center}
			\begin{tabular}{lcc}
				\toprule
				$M$            & common-index edges & upper bound for $E$ \\
				\midrule
				$\{12,14,25\}$ & 2                  & 3                   \\
				$\{12,23,45\}$ & 7                  & 7                   \\
				$\{12,24,25\}$ & 2                  & 3                   \\
				$\{12,24,35\}$ & 2                  & 3                   \\
				$\{12,24,45\}$ & 4                  & 4                   \\
				$\{14,24,25\}$ & 1                  & 2                   \\
				$\{24,25,45\}$ & 2                  & 2                   \\
				\bottomrule
			\end{tabular}
		\end{center}
		Thus $E$ is never $5$ or $6$, contradicting \eqref{eq:E14}. This excludes $N=14$, and therefore $N=16$.
	\end{proof}
	
	From now on
	\begin{equation}\label{eq:16}
		|S|=N=16.
	\end{equation}
	Equations \eqref{eq:abc-general} and \eqref{eq:ac-general} become
	\begin{equation}\label{eq:abc}
		a-c=10-2e,\qquad b=20-2e-2a.
	\end{equation}
	No point of $O$ is adjacent to $0$, and every sum point has exactly three neighbours in $O$.

	\section{Every point lies in at least two unit triangles}\label{sec:triangles}
	
	The following lemma replaces Cases 1 and 2 (2.1--2.5) of \cite{H}.
	
	\begin{lemma}\label{lem:two}
		$t(v)\ge 2$ for every $v\in S$.
	\end{lemma}
	
	\begin{proof}
		Let $v=0$ and suppose $e\le 1$; if $e=1$, let $z_1\adj z_2$. For a sum point $\sigma=s_{ij}$ let $s(\sigma)$ be the number of its neighbours among the sum points $s_{kl}$ with $|\{k,l\}\cap\{i,j\}|=1$, let $q(\sigma)$ be the number of its neighbours among the sum points $s_{kl}$ with $\{k,l\}\cap\{i,j\}=\emptyset$, and let $n(\sigma)$ be the number of its neighbours of type B or C. Then
		\begin{equation}\label{eq:sqn}
			s(\sigma)+q(\sigma)+n(\sigma)=3 .
		\end{equation}
		If $q(\sigma)\ge 2$, then $\sigma$ is a neighbour of two of the points $s_{rs},s_{rt},s_{st}$, where $\{r,s,t\}$ is the complement of $\{i,j\}$; these two share an index, so Lemma~\ref{lem:rules}(ii) yields $z_i+z_j=z_r+z_s+z_t$. By Lemma~\ref{lem:B},
		\begin{equation}\label{eq:q}
			q(\sigma)\le 1\quad\text{for all sum points $\sigma$ with at most one exception.}
		\end{equation}
		
		\emph{Case $e=0$.} By \eqref{eq:abc}, $a=10$ and $b=c=0$. By Lemma~\ref{lem:rules}(i), $s(\sigma)=0$, hence $q(\sigma)=3$ for all ten sum points, contradicting \eqref{eq:q}.
		
		\emph{Case $e=1$.} By \eqref{eq:abc}, $(a,b,c)$ is $(9,0,1)$ or $(8,2,0)$. By Lemma~\ref{lem:rules}(i), $s(\sigma)=1$ if $\sigma\in\{s_{1k},s_{2k}\}$ for some $k\ge 3$ and both $s_{1k},s_{2k}$ are sum points, and $s(\sigma)=0$ otherwise.
		
		Let $(a,b,c)=(9,0,1)$. Each $z_k$, $k\ge3$, has four neighbours in $O$, all of them sum points, so all nine $s_{ij}$ with $\{i,j\}\ne\{1,2\}$ are sum points. For $\sigma\in\{s_{34},s_{35},s_{45}\}$ we have $s(\sigma)=0$ and $n(\sigma)\le 1$, hence $q(\sigma)\ge 2$ for three sum points, contradicting \eqref{eq:q}.
		
		Let $(a,b,c)=(8,2,0)$. Each of the two points of type B has four neighbours besides its neighbour in $Z$, so $\sum_\sigma n(\sigma)\le 8$. On the other hand, $\sum_\sigma s(\sigma)\le 6$, and by \eqref{eq:q} we have $\sum_\sigma q(\sigma)\le 7+3$. By \eqref{eq:sqn}, $\sum_\sigma n(\sigma)\ge 24-6-10=8$. Hence equality holds everywhere: all six points $s_{1k},s_{2k}$ ($k\ge 3$) are sum points, and there is a sum point $\sigma_0$ with $q(\sigma_0)=3$, $s(\sigma_0)=0$. Since $z_1$ has only three neighbours in $O$, we get $s_{12}\notin O$. As $s(\sigma_0)=0$, we have $\sigma_0=s_{kl}$ with $k,l\ge3$; but one of the three sums with index set disjoint from $\{k,l\}$ is $s_{12}\notin O$, so $q(\sigma_0)\le 2$, a contradiction.
	\end{proof}
	
	\begin{corollary}\label{cor:three}
		There is a point $P\in S$ with $t(P)\ge 3$. For such a point, either four of its neighbours are consecutive vertices of a regular hexagon with centre $P$, or $t(P)=3$ and the unit distances among its neighbours form a path with three vertices and a disjoint edge.
	\end{corollary}
	
	\begin{proof}
		If $t(v)=2$ for all $v$, then counting incidences gives $3T=2\cdot 16=32$ for the number $T$ of unit triangles, which is impossible (this replaces Case 3 of \cite{H}). The second statement follows from (P2).
	\end{proof}

	\section{Four consecutive hexagon neighbours}\label{sec:hexagon}
	
	We first dispose of the two local configurations in which four or five consecutive vertices of a regular hexagon occur around one point. The argument below replaces Cases 4 and 5 of \cite{H}; no six-point core needs to be excluded in advance.
	
	\begin{lemma}\label{lem:fourtri}
		No point of $S$ lies in four unit triangles.
	\end{lemma}
	
	\begin{proof}
		Suppose that $0$ has five neighbours $z_1,\ldots,z_5$ which are consecutive vertices of a regular hexagon. For $i=1,\ldots,4$ put
		\[
		s_i=z_i+z_{i+1}.
		\]
		Any point of $O$ with two neighbours in $Z$ is one of the points $s_i$. Indeed, for two neighbours separated by two hexagon steps the second circle intersection is the intermediate vertex, for opposite neighbours it is $0$, and the sum $z_1+z_5$ is the sixth hexagon vertex and cannot belong to $S$ because $0$ already has five neighbours.
		
		Let $a$ be the number of the points $s_i$ that belong to $S$, and let $c$ be the number of points of type C. Since $e=4$, \eqref{eq:abc} gives
		\begin{equation}\label{eq:e4ac}
			a=c+2.
		\end{equation}
		Fix a present point $s_i$. It already has the neighbours $z_i,z_{i+1}$. If another neighbour of $s_i$ has a neighbour in $Z$, the two-circle rule (P1), applied to $s_i$ and the relevant $z_j$, shows that this point is one of $2z_i,2z_{i+1}$ (apart from $z_i,z_{i+1}$ themselves). Indeed the remaining distances from $s_i$ to the five points of $Z$ are $2$, $2$, and $\sqrt7$. Hence at least one of the three remaining neighbours of $s_i$ is of type C.
		
		These type C neighbours are distinct for different present $s_i$. For $|i-j|\ge2$ we have $|s_i-s_j|>2$. If $j=i+1$, then $|s_i-s_{i+1}|=\sqrt3$, and their two common unit-distance points are $z_{i+1}$ and $2z_{i+1}$, neither of type C. Thus every present $s_i$ requires its own type C neighbour, and $c\ge a$, contradicting \eqref{eq:e4ac}.
	\end{proof}
	
	\begin{lemma}[Four-chain supply lemma]\label{lem:four-supply}
		Suppose that $0$ has four consecutive hexagon neighbours $z_1,z_2,z_3,z_4$ and a fifth neighbour $z_5$ which is adjacent to none of them. Then, after ordering the four consecutive directions,
		\[
		s_{12},s_{23},s_{34}\notin S, \qquad s_{15},s_{25},s_{35},s_{45}\in S.
		\]
		Moreover the remaining six points of $O$ are all of type B and span nine unit distances.
	\end{lemma}
	
	\begin{proof}
		Normalize
		\[
		z_1=1,\qquad z_2=\w,\qquad z_3=\w^2,\qquad z_4=-1.
		\]
		Since $z_5$ is distinct from and adjacent to none of $z_1,\ldots,z_4$, it is not a sixth root of unity: the two remaining sixth roots $\w^4,\w^5$ are adjacent to $z_4,z_1$, respectively. Hence no sum involving $z_5$ can be $0$ or a point of $Z$; otherwise $z_5$ would be the opposite of, or a unit difference of, two sixth roots and hence itself a sixth root. Also $|s_{12}|=|s_{23}|=|s_{34}|=\sqrt3$, so these three sums lie outside $Z\cup\{0\}$. Together with the preceding uniqueness observation, the seven sums below are distinct and lie outside $Z\cup\{0\}$. Since $s_{13}=z_2$, $s_{24}=z_3$ and $s_{14}=0$, the only possible sum points in $O$ are
		\begin{equation}\label{eq:seven-fourchain}
			12,23,34,15,25,35,45,
		\end{equation}
		where, here and below, $ij$ abbreviates $s_{ij}$. Let $M$ be the set of missing points from \eqref{eq:seven-fourchain}, put $m=|M|$, and let $m_i$ count the members of $M$ containing the index $i$. Since $a=7-m$ and $a-c=4$, we have
		\begin{equation}\label{eq:fourchain-mc}
			m+c=3,\qquad b=2m.
		\end{equation}
		Counting the missing sum neighbours at the five points $z_i$ gives
		\begin{equation}\label{eq:fourchain-B}
			|B_1|=1+m_1,\quad |B_2|=m_2-1,\quad |B_3|=m_3-1,
			\quad |B_4|=1+m_4,\quad |B_5|=m_5.
		\end{equation}
		In particular $m_2,m_3\ge1$.
		
		Among the seven possible sum points, the only unit distances are
		\begin{equation}\label{eq:fourchain-path}
			15\adj25\adj35\adj45.
		\end{equation}
		For pairs with a common index this follows from Lemma~\ref{lem:rules}(i). Among disjoint pairs, $|s_{12}-s_{34}|=3$. For each of the other six pairs, after multiplication by a unit of $\mathbb Z[\w]$, the difference has the form $z_5+u$, where $u\in\mathbb Z[\w]$ has squared length $4$ or $7$. Lemma~\ref{lem:eisenstein-circle} therefore excludes a further unit distance. Let $E$ denote the number of edges induced by the present sum points. Thus
		\begin{equation}\label{eq:fourchain-E}
			E\le3.
		\end{equation}
		
		We also need one elementary supply observation. A type B point can be adjacent to two of the seven sums only in one of the following two ways:
		\[
		B_2:\ \{12,23\},\qquad B_3:\ \{23,34\}.
		\]
		The corresponding points are $2z_2$ and $2z_3$. We verify this exhaustively. If two sums share an index, Lemma~\ref{lem:rules}(ii) gives their second common unit-distance point. Besides the two displayed pairs, the possibilities are as follows. The pairs
		\[
		\begin{split}
			& \{12,15\},\{12,25\},\{15,25\}, \\
			& \{23,25\},\{23,35\},\{25,35\}, \\
			& \{34,35\},\{34,45\},\{35,45\}
		\end{split}
		\]
		have second common point $z_5+z_i+z_{i+1}$ for $i=1,2,3$, respectively. For any $r\le4$, adjacency of such a point to $z_r$ is an equation $|z_5+u|=1$ with $u\in\mathbb Z[\w]$ and $|u|\in\{1,2,\sqrt7\}$; adjacency to $z_5$ would require $|z_i+z_{i+1}|=1$, which is false. Lemma~\ref{lem:eisenstein-circle} therefore shows that these points are not of type B. The three remaining shared-index pairs have second common point $25$, $z_5$, or $35$, hence a sum point or a point of $Z$.
		
		For disjoint index pairs, $12$ and $34$ are at distance $3$ and have no common unit-distance point. For each of the other six pairs, Lemma~\ref{lem:rules}(vi), applied to both sums, restricts the possible unique $Z$-neighbour $z_r$ of a type B point. In the order of the table below the admissible indices are
		\[
		\{1,5\},\quad \{1,2,4,5\},\quad \{3,5\},\quad \{2,5\},\quad \{1,3,4,5\},\quad \{4,5\}.
		\]
		Applying Lemma~\ref{lem:rules}(iv) to one of the sums then leaves two candidates. After division by a unit of $\mathbb Z[\w]$, the second adjacency has the form $|z_5+u|=1$; the possible squared lengths of $u$ are
		\[
		\begin{array}{c|c}
			\{12,35\} & 1,7     \\
			\{12,45\} & 3,4,7,9 \\
			\{23,15\} & 1,7     \\
			\{23,45\} & 1,7     \\
			\{34,15\} & 3,4,7,9 \\
			\{34,25\} & 1,7.
		\end{array}
		\]
		Lemma~\ref{lem:eisenstein-circle} excludes every case. Thus the two displayed double supplies are the only possibilities.
		
		In fact the two exceptional double supplies cannot occur for present sum points. If $B_2\ne\emptyset$, then \eqref{eq:fourchain-B} gives $m_2\ge2$; but if both $12$ and $23$ were present, among the three possible sums $12,23,25$ at most one could be missing, so $m_2\le1$. The argument for $B_3$ is identical. Consequently every type B point is adjacent to at most one present sum point. By Lemma~\ref{lem:five-sums}, every type C point is adjacent to at most four present sum points.
		
		Let $F$ be the number of unit distances from the sum points to the points of types B and C. Each sum point has degree three inside $O$, hence
		\[
		F=3a-2E\ge 3(4+c)-6=6+3c.
		\]
		On the other hand the preceding supply bounds and $b=6-2c$ from \eqref{eq:fourchain-mc} give
		\[
		F\le b+4c=6+2c.
		\]
		Therefore $c=0$. Equality must hold throughout, so $E=3$, $a=4$, $b=6$. By \eqref{eq:fourchain-path}, $E=3$ forces all four points $15,25,35,45$ to be present. Since there are exactly four sum points, the three points $12,23,34$ are absent.
		
		The four sum points span three unit distances, so the number of unit distances from them to the six type B points is $4\cdot3-2\cdot3=6$. The six type B points have total degree $24$ inside $O$, hence they span
		\[
		\frac{24-6}{2}=9
		\]
		unit distances.
	\end{proof}
	
	\begin{lemma}\label{lem:hex}
		No point of $S$ has four neighbours which are consecutive vertices of a regular hexagon centred at that point.
	\end{lemma}
	
	\begin{proof}
		Suppose first that the fifth neighbour continues the hexagon path. Then the point lies in four unit triangles, contrary to Lemma~\ref{lem:fourtri}. Thus the fifth neighbour is isolated from the four consecutive ones, and Lemma~\ref{lem:four-supply} applies.
		
		Let
		\[
		S_1=\{s_{15},s_{25},s_{35},s_{45}\}
		\]
		and let $B$ be the six type B points. By Lemma~\ref{lem:four-supply}, $B$ spans nine unit distances. Fact~\ref{fact:six} gives four possibilities. The configuration $T_1$ contains a point with five consecutive hexagon neighbours and is excluded by Lemma~\ref{lem:fourtri}. In $T_2$, the point $(0,0)$ has the four consecutive neighbours $(1,0),(0,1),(-1,1),(-1,0)$ and the corresponding middle completion $(-1,2)$ also belongs to $T_2$. In $T_3$ the same four neighbours occur and the completion $(1,1)$ belongs to $T_3$. For either $T_2$ or $T_3$, the fifth neighbour in $S$ either creates a fourth local unit triangle, already excluded by Lemma~\ref{lem:fourtri}, or is isolated from the four-chain, in which case Lemma~\ref{lem:four-supply} says that all three completions must be absent. Hence $B$ must be a prism.
		
		Thus every point of $B$ has degree three inside $B$, and since it has one neighbour in $Z$, it has exactly one neighbour in $S_1$. Moreover any two points of the prism have distance less than $2$.
		
		Write $z_5=\zeta$. The type B classes are
		\[
		|B_1|=|B_4|=2,\qquad |B_2|=|B_3|=1,\qquad B_5=\emptyset.
		\]
		The path \eqref{eq:fourchain-path} shows that $s_{15},s_{45}$ need two neighbours in $B$, while $s_{25},s_{35}$ need one each. By Lemma~\ref{lem:rules}(iii), no point of $B_1$ is adjacent to $s_{25}$, and no point of $B_4$ is adjacent to $s_{35}$.
		
		Not both points of $B_1$ can be adjacent to $s_{15}$. Otherwise Lemma~\ref{lem:rules}(iv) gives
		\[
		B_1=\{z_1+\zeta\w,z_1+\zeta\wb\}.
		\]
		Every point of $B_4$ must then be assigned to $s_{25}$ or $s_{45}$; since $s_{25}$ has only one neighbour in $B$, some point of $B_4$ is adjacent to $s_{45}$ and has the form $z_4+\zeta\w^\varepsilon$. It is at distance $|z_4-z_1|=2$ from one of the displayed points of $B_1$, contradicting the prism property. The symmetric argument applies to $B_4$ and $s_{45}$. Consequently
		\begin{equation}\label{eq:cross-needs}
			\begin{split}
				& \text{some point of }B_1\text{ is adjacent to }s_{35}\text{ or }s_{45}, \\
				& \text{some point of }B_4\text{ is adjacent to }s_{15}\text{ or }s_{25}.
			\end{split}
		\end{equation}
		
		The preceding conclusions, including \eqref{eq:cross-needs}, were derived for an arbitrary centre satisfying the hypotheses of Lemma~\ref{lem:four-supply}. The point $\zeta$ has the neighbours $0$ and $\zeta+z_i$, $i=1,\ldots,4$; the latter are four consecutive hexagon vertices about $\zeta$, while $0$ is isolated from that four-chain. After translation by $-\zeta$ we therefore obtain the same normal form with the fifth direction $\zeta$ replaced by $-\zeta$. Hence, after interchanging the roles of the two centres if necessary, we may write $\zeta=e^{i\theta}$ with $0<\theta<\pi$. Since $z_4=-z_1$, a common unit-distance point can exist only if the two centres have distance at most $2$, and
		\begin{align*}
			|s_{35}-z_1|^2 & =4+2\sqrt3\cos(\theta-5\pi/6), &
			|s_{45}-z_1|^2 & =5-4\cos\theta,                  \\
			|s_{25}-z_4|^2 & =4+2\sqrt3\cos(\theta-\pi/6),  &
			|s_{15}-z_4|^2 & =5+4\cos\theta.
		\end{align*}
		For $0<\theta<\pi$, the inequality $|s_{35}-z_1|\le2$ holds only for $\theta\le\pi/3$, while $|s_{45}-z_1|\le2$ implies $\cos\theta\ge1/4$. Thus the first line of \eqref{eq:cross-needs} implies $\cos\theta\ge1/4$. Similarly $|s_{25}-z_4|\le2$ holds only for $\theta\ge2\pi/3$, while $|s_{15}-z_4|\le2$ implies $\cos\theta\le-1/4$. Hence the second line implies $\cos\theta\le-1/4$. This contradiction proves the lemma.
	\end{proof}
	
	\begin{corollary}\label{cor:six-prism}
		Every six-point subset of $S$ which spans nine unit distances is a prism.
	\end{corollary}
	
	\begin{proof}
		By Fact~\ref{fact:six} the other three possibilities are $T_1,T_2,T_3$. Each contains a point with four consecutive hexagon neighbours, contrary to Lemma~\ref{lem:hex}.
	\end{proof}

	\section{A path and a disjoint edge}\label{sec:path}
	
	By Corollary~\ref{cor:three} and Lemma~\ref{lem:hex}, there is a point, say $0$, whose neighbours satisfy
	\[
	z_1\adj z_2\adj z_3,\qquad z_2=z_1+z_3,\qquad z_4\adj z_5,
	\]
	with no further unit distances among them. This is Case 6 of \cite{H}.
	
	\begin{lemma}\label{lem:case6}
		The path-and-edge configuration is impossible.
	\end{lemma}
	
	\begin{proof}
		After a rotation and, if necessary, a reflection, normalize
		\[
		z_1=1,\qquad z_2=\w,\qquad z_3=\w^2.
		\]
		We have $e=3$, hence $a=c+4$ and $b=6-2c$. Neither $z_4$ nor $z_5$ is a sixth root of unity. They are distinct from $1,\w,\w^2$ by hypothesis; moreover $\w^3$ and $\w^5$ are adjacent to $z_3$ and $z_1$, respectively, and if one of $z_4,z_5$ were $\w^4$, the other, being a unit vector at distance $1$ from it, would be $\w^3$ or $\w^5$. Thus none of
		\[
		14,24,34,15,25,35
		\]
		coincides with $0$ or a point of $Z$: such a coincidence would make $z_4$ or $z_5$ a sixth root. Also $|s_{45}|=\sqrt3$, so the same holds for $45$. The seven positions are also distinct by the uniqueness observation in Section~2.1. Moreover $s_{13}=z_2$, while $s_{12},s_{23}\notin S$: otherwise $z_2$ would have four consecutive hexagon neighbours, contrary to Lemma~\ref{lem:hex}.
		
		Let $M$ be the set of missing points among the seven displayed sums, $m=|M|$, and let $m_i$ count the members of $M$ containing $i$. Since $a=7-m$,
		\begin{equation}\label{eq:ladder-mc}
			m+c=3,\qquad b=2m.
		\end{equation}
		If $B_i$ is the set of type B points whose unique neighbour in $Z$ is $z_i$, then counting at $z_i$ gives
		\begin{equation}\label{eq:ladder-beta}
			|B_1|=1+m_1,\quad |B_2|=m_2,\quad |B_3|=1+m_3,
			\quad |B_4|=m_4-1,\quad |B_5|=m_5-1.
		\end{equation}
		Thus $m_4,m_5\ge1$. For $m=1,2,3$ there are respectively $1,15,33$ subsets $M$ satisfying these necessary conditions. Under the index relabellings $1\leftrightarrow3$ and $4\leftrightarrow5$, these give respectively $1,6,11$ classes; the relabellings preserve the local path-and-edge pattern. We examine all three values of $m$ below.
		
		The six sums $14,24,34,15,25,35$ form a $2\times3$ ladder:
		\begin{equation}\label{eq:ladder-edges}
			14\adj15,\quad24\adj25,\quad34\adj35,
			\qquad14\adj24\adj34,\qquad15\adj25\adj35.
		\end{equation}
		The candidate sum $45$ is isolated from the other sum points (Figure~\ref{fig:ladder}). These are all unit distances among the seven sums. Indeed, the common-index pairs are given by Lemma~\ref{lem:rules}(i), while a disjoint pair would contradict Lemma~\ref{lem:rules}(iii) using one of $z_1\adj z_2$, $z_2\adj z_3$, $z_4\adj z_5$. 
		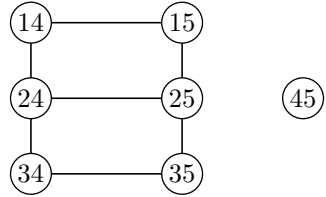
\begin{figure}[htbp]
			\centering
			\begin{tikzpicture}[x=1cm,y=1cm]
				\draw[line width=0.6pt] (0,2) -- (0,1) -- (0,0);
				\draw[line width=0.6pt] (2.0,2) -- (2.0,1) -- (2.0,0);
				\draw[line width=0.6pt] (0,2) -- (2.0,2);
				\draw[line width=0.6pt] (0,1) -- (2.0,1);
				\draw[line width=0.6pt] (0,0) -- (2.0,0);
				\foreach \x/\y/\n in {0/2/14,0/1/24,0/0/34,2/2/15,2/1/25,2/0/35,3.6/1/45}{
					\filldraw[fill=white,draw=black,line width=0.5pt] (\x,\y) circle (0.27);
					\node[font=\small] at (\x,\y) {$\n$};
				}
			\end{tikzpicture}
			\caption{The unit-distance adjacency pattern of the seven candidate sum points in the path-and-edge case. The diagram is schematic, not a Euclidean realization of their positions; $45$ is isolated.}
			\label{fig:ladder}
		\end{figure}
		
		For a present sum point $\sigma$, let $\lambda(\sigma)$ be the number of its present ladder neighbours. Then it needs exactly
		\[
		n(\sigma)=3-\lambda(\sigma)
		\]
		neighbours of types B and C. Put
		\begin{equation}\label{eq:DE}
			D=\sum_{\sigma}n(\sigma)=3(7-m)-2E,
		\end{equation}
		where $E$ is the number of present ladder edges.
		
		We shall repeatedly use three local supply rules, all direct consequences of Lemma~\ref{lem:rules}.
		\begin{itemize}
			\item[(G1)] The point $s_{24}$ has no neighbour in $B_1\cup B_3\cup B_4\cup B_5$ and at most one in $B_2$; the same holds for $s_{25}$.
			\item[(G2)] The point $s_{14}$ has no neighbour in $B_2\cup B_5$ and at most one in each of $B_1,B_4$. For $s_{15}$ the excluded classes are $B_2,B_4$ and there is at most one neighbour in each of $B_1,B_5$; the statements for $s_{34},s_{35}$ follow by $1\leftrightarrow3$.
			\item[(G3)] No point of type B or C is adjacent to both $s_{14}$ and $s_{34}$, or to both $s_{15}$ and $s_{35}$.
		\end{itemize}
		For (G1), use Lemma~\ref{lem:rules}(iii) with the adjacent pairs $z_1\adj z_2$, $z_3\adj z_2$, $z_5\adj z_4$, and use (iv) for the remaining classes. Rule (G2) is the same argument with the indices shifted. Rule (G3) is Lemma~\ref{lem:rules}(v), since $z_1+z_3=z_2$.
		
		\emph{Case $m=3$.} Here $c=0$ and $b=6$. Up to the symmetries $1\leftrightarrow3$ and $4\leftrightarrow5$, there are eleven possible sets $M$. Rule (G1) immediately excludes eight of them; the required number of type B neighbours of the indicated central sum is larger than $|B_2|=m_2$:
		\[
		\begin{array}{c|c|c|c}
			M            & \sigma & n(\sigma) & |B_2| \\ \hline
			\{45,14,34\} & 24     & 2         & 0     \\
			\{45,14,15\} & 24     & 1         & 0     \\
			\{45,14,25\} & 24     & 2         & 1     \\
			\{45,14,35\} & 24     & 1         & 0     \\
			\{14,24,15\} & 25     & 2         & 1     \\
			\{14,24,35\} & 25     & 2         & 1     \\
			\{14,34,15\} & 25     & 1         & 0     \\
			\{14,34,25\} & 24     & 3         & 1
		\end{array}
		\]
		Notice that $M=\{45,14,25\}$ is the missing-sum pattern not explicitly treated among the subcases in the online copy of \cite{H} consulted here (see Remark~\ref{rem:missing-case}).
		
		Three patterns remain.
		
		If $M=\{14,24,25\}$, then $|B_1|=2$, $|B_3|=1$ and $B_5=\emptyset$. The point $s_{15}$ has no ladder neighbour and therefore needs three type B neighbours. By (G2) it has none in $B_2\cup B_4$, at most one in $B_1$, and there is only one point in $B_3$. Thus at most two are available, a contradiction.
		
		If $M=\{45,24,25\}$, then $|B_1|=|B_3|=|B_4|=|B_5|=1$ and $|B_2|=2$. Each of $s_{14},s_{15},s_{34},s_{35}$ needs two type B neighbours, for a total of eight incidences. The two points of $B_2$ supply none of them. By (G3) the points of $B_4$ and $B_5$ supply at most one each, and the points of $B_1$ and $B_3$ at most two each. Thus at most six incidences are possible.
		
		Finally let $M=\{45,14,24\}$. Then $D=6$, so the six type B points span $(24-6)/2=9$ unit distances. By Corollary~\ref{cor:six-prism} they form a prism, and hence every one of them has exactly one neighbour among the sum points. Formula~\eqref{eq:ladder-beta} gives $|B_4|=2$. By Lemma~\ref{lem:rules}(iii), neither point of $B_4$ is adjacent to $15,25,35$, so both must be adjacent to $34$. But the unit circles centred at $z_4$ and $s_{34}$ have only two common points; one is the position $z_2+z_4=s_{24}$, which, if occupied by a point of $S$, would have the two neighbours $z_2,z_4$ in $Z$ and hence could not be of type B. Thus at most one common point can be of type B, a contradiction.
		
		\emph{Case $m=2$.} Now $c=1$ and $b=4$. Let $Q$ be the point of type C, let $k$ be the number of unit distances among the four type B points, and let $x$ be the number of neighbours of $Q$ among them. Since $Q$ has $5-x$ sum neighbours, counting degrees at the four type B points gives
		\begin{equation}\label{eq:kx-short}
			2x=21-D-2k.
		\end{equation}
		We first claim that
		\begin{equation}\label{eq:D-not-57}
			D\notin\{5,7\}.
		\end{equation}
		By (P3), $k\le5$, and clearly $x\le4$. Equation~\eqref{eq:kx-short} leaves only
		\[
		\begin{array}{c|c}
			D=5 & (k,x)=(4,4),(5,3),       \\
			D=7 & (k,x)=(3,4),(4,3),(5,2).
		\end{array}
		\]
		If $x=4$, the four type B points lie on the unit circle about $Q$, so (P2) gives $k\le3$; equality would make them four consecutive hexagon neighbours of $Q$, contrary to Lemma~\ref{lem:hex}. If $k=5$, the four type B points form a diamond, say with unit triangles $pqs$ and $rqs$. An external point cannot be adjacent to three vertices of the diamond: three containing a unit triangle are impossible, while the other triples contain the opposite pair $p,r$, whose only common unit-distance points are $q,s$. This excludes $(5,3)$. For $(5,2)$, the two neighbours of $Q$ cannot be $p,r$ or $q,s$, so by symmetry take $Q\adj p,q$. Put $d=p-q$ and write $s-q=d\w^\varepsilon$. Then $r-q=d\w^{2\varepsilon}$ and the second common unit-distance point of $p,q$ is $Q-q=d\w^{-\varepsilon}$. Thus $q$ has the four consecutive neighbours $Q,p,s,r$, again impossible.
		
		It remains $(k,x)=(4,3)$. By (P3), the four type B points form either a rhombus or a unit triangle $pqr$ with a pendant point $s\adj p$. In a rhombus any three vertices contain an opposite pair, whose two common unit-distance points are the other two vertices, so $Q$ cannot be external. In the triangle-pendant case the three neighbours of $Q$ cannot be $p,q,r$. If $Q\adj s,q,r$, then $Q=q+r-p$; since $s$ is adjacent to both $p$ and $Q$, (P1) forces $s\in\{q,r\}$. If $Q\adj s,p,q$, then $Q=p+q-r$ and (P1) gives $s=2p-r$, so $p$ has four consecutive hexagon neighbours $r,q,Q,s$; the case $Q\adj s,p,r$ is symmetric. This proves \eqref{eq:D-not-57}.
		
		There are six sets $M$ up to symmetry. Three have $D\in\{5,7\}$:
		\[
		M=\{45,14\}\ (D=5),\qquad M=\{45,24\}\ (D=7),\qquad M=\{14,15\}\ (D=7),
		\]
		and are excluded by the claim.
		
		Let $M=\{14,25\}$. Then $|B_1|=2$, $|B_2|=|B_3|=1$ and $B_4=B_5=\emptyset$. We have $n(s_{15})=3$ and $n(s_{35})=2$. The point of $B_2$ is adjacent to neither of these sums by (G2), and by (G3) each of the remaining four points of types B and C is adjacent to at most one of them. Thus five required incidences would have to be supplied by at most four.
		
		Let $M=\{24,25\}$. Then $|B_1|=|B_3|=1$ and $|B_2|=2$. Each of $s_{14},s_{15},s_{34},s_{35}$ needs two neighbours of types B and C. The points of $B_2$ supply none by (G2), while each of the remaining three points supplies at most two by (G3). Hence eight required incidences would have to be supplied by at most six.
		
		Finally let $M=\{14,35\}$. Then $|B_1|=|B_3|=2$ and $B_2=B_4=B_5=\emptyset$. By (G1), the unique non-ladder neighbour required by each of $s_{24},s_{25}$ must be $Q$. Lemma~\ref{lem:rules}(ii) gives
		\[
		Q=z_2+z_4+z_5,
		\]
		so $Q\adj s_{45}$. Moreover $|Q-s_{15}|=|z_3+z_4|\ne1$ and $|Q-s_{34}|=|z_1+z_5|\ne1$, because $s_{34},s_{15}$ are sum points and therefore are not adjacent to $0$. Consequently each point of
		\[
		X=\{Q,s_{15},s_{34},s_{45}\}
		\]
		has exactly two neighbours in the four-point set $B=B_1\cup B_3$. Counting degrees at $B$ shows that $B$ spans four unit distances.
		
		Put
		\[
		w_1=z_1+z_4+z_5,\qquad w_3=z_3+z_4+z_5.
		\]
		By (P1) and Lemma~\ref{lem:rules}(ii), the common unit-distance points of the following pairs are
		\[
		\begin{array}{c|c}
			Q,s_{45}      & w_1,w_3     \\
			s_{45},s_{15} & z_5,w_1     \\
			s_{45},s_{34} & z_4,w_3     \\
			Q,s_{15}      & s_{25},w_1  \\
			Q,s_{34}      & s_{24},w_3.
		\end{array}
		\]
		For $y\in B$, let $d(y)$ be the number of its neighbours in $X$; then $\sum_{y\in B}d(y)=8$. If $y\notin\{w_1,w_3\}$, the table gives $d(y)\le2$, and $d(y)\le1$ when $y\adj Q$ or $y\adj s_{45}$. The point $w_1$ is adjacent to $Q,s_{45},s_{15}$ but not to $s_{34}$; indeed $|z_1-z_3+z_5|=1$ would force $z_5\in\{z_3,-z_1\}$, contrary to the path-and-edge hypothesis. Thus $d(w_1)=3$ if $w_1\in B$, and symmetrically $d(w_3)=3$ if $w_3\in B$.
		
		If neither $w_1$ nor $w_3$ belongs to $B$, the two neighbours of $Q$ and the two neighbours of $s_{45}$ in $B$ are four distinct points with $d=1$, so $\sum d=4$. If exactly one of $w_1,w_3$ belongs to $B$, it is a neighbour of both $Q$ and $s_{45}$; their other neighbours in $B$ are distinct points with $d=1$, and hence $\sum d\le3+1+1+2=7$. Therefore $w_1,w_3\in B$. By (G1), no point of $B=B_1\cup B_3$ is adjacent to $s_{24}$ or $s_{25}$. Hence each of $w_1,w_3$ has exactly the three neighbours in $X$ listed above outside $B$, and therefore exactly one neighbour inside $B$. Since $|w_1-w_3|=|z_1-z_3|=\sqrt3$, these two vertices are nonadjacent, and a four-vertex graph in which two nonadjacent vertices have degree one has at most three edges. This contradicts the four edges already counted in $B$.
		
		\emph{Case $m=1$.} The conditions $m_4,m_5\ge1$ force $M=\{45\}$. Then $c=2$, $b=2$, and \eqref{eq:DE} gives $D=4$. The two type B and two type C points have total degree
		\[
		2\cdot4+2\cdot5=18
		\]
		inside $O$. After the four edges to the sum points are removed, these four points would span $(18-4)/2=7$ unit distances, contradicting (P3).
		
		All possibilities are excluded.
	\end{proof}
	
	\begin{proof}[Proof of Theorem~\ref{thm:main}]
		Suppose that a $5$-regular set $S$ has $|S|\le16$. By Proposition~\ref{prop:reduction}, $|S|=16$. Corollary~\ref{cor:three} gives a point with $t\ge3$, which contradicts Lemma~\ref{lem:hex} or Lemma~\ref{lem:case6}. Hence no $5$-regular set has at most $16$ points. Since every $5$-regular graph has even order, $p(5)\ge18$. Together with the upper bound $p(5)\le18$, this proves the theorem.
	\end{proof}

	\section{Remarks}\label{sec:remarks}
	
	\begin{remark}\label{rem:missing-case}
		Correspondence with \cite{H}. Lemma~\ref{lem:two} replaces Cases 1 and 2, and Corollary~\ref{cor:three} replaces Case 3. Lemmas~\ref{lem:fourtri}--\ref{lem:hex} replace Cases 4 and 5 by direct local incidence estimates. Lemma~\ref{lem:case6} treats Case 6 by the three values $m=3,2,1$ and local rules on a $2\times3$ ladder. In the online PDF copy of \cite{H} available to us, the numbering passes directly from 6.6 to 6.8 on pages 248--249. These two subcases have missing-sum sets $\{45,14,15\}$ and $\{45,14,35\}$, respectively, in our notation. We could not locate an explicit treatment of the intervening pattern $M=\{45,14,25\}$ in that copy. We have not examined a physical copy of the 1987 volume and cannot determine whether the apparent omission was present in the original publication or arose during digitization. Our $m=3$ table excludes the configuration without a further appeal to Fact~\ref{fact:six}.
	\end{remark}
	
	\begin{remark}
		The only external finite enumeration used in the proof is Harary's 1969 atlas of the six-vertex graphs; the short geometric elimination needed from that atlas is included in Fact~\ref{fact:six}. No result subsequent to Harborth's argument is used.
	\end{remark}

	\section{Further questions}\label{sec:questions}
	
	The next case is $p(6)$. A generic Minkowski sum of three unit triangles gives $p(6)\le p(2)^3=27$. Is the bound sharp? Equivalently, can one exclude a set of at most $26$ points in which every point has exactly six unit-distance neighbours?
	
	Can the local incidence counts, circle-intersection identities and supply rules used here be extended to this problem? Unlike in (P2), six neighbours may occupy an entire regular hexagon, so local reductions must account for cycles as well as paths.
	
	What can be said about the structure at $p(5)=18$? The generic Minkowski sum $\{0,u\}+\Delta_1+\Delta_2$ of a unit segment and two unit triangles has unit-distance graph $K_2\mathbin{\square}K_3\mathbin{\square}K_3$, with $t(v)=2$ at every vertex and $12$ triangles altogether. Which other local triangle counts and neighbourhood patterns can occur in an $18$-point $5$-regular set? Could the constraints developed here help classify these extremal configurations? The $16$-point exclusions do not answer these questions.

\end{document}